\documentclass[11pt]{amsart}
\usepackage{mathrsfs,amssymb}
\usepackage{verbatim}
\usepackage{graphicx}

\usepackage{graphicx,latexsym}
\usepackage{hyperref}

\newtheorem{innercustomthm}{Hypothesis}  
\newenvironment{customthm}[1]
  {\renewcommand\theinnercustomthm{#1}\innercustomthm}
  {\endinnercustomthm}

\newtheorem{theorem}{Theorem}[section]
\newtheorem{proposition}[theorem]{Proposition}
\newtheorem{lemma}[theorem]{Lemma}
\newtheorem{corollary}[theorem]{Corollary}
\theoremstyle{remark}

 \numberwithin{equation}{section}

\renewcommand{\Re}{\operatorname{Re}}
\renewcommand{\Im}{\operatorname{Im}}

\def \R {{\mathbb R}}

\def \Z {{\mathbb Z}}

\newcommand{\area}{\operatorname{area}}

\newcommand{\intinf}{\int_{-\infty}^\infty}

\newcommand{\betaomega}{\mathcal B_\Omega}

\newcommand{\dd}{\mathrm{d}}

\newcommand{\oscI}{\mathcal I} 

\begin{document}

\title{Prime lattice points in ovals}
\author{Bingrong Huang and Ze\'ev Rudnick}
\address{Raymond and Beverly Sackler School of Mathematical Sciences, Tel Aviv University, Tel Aviv 69978, Israel}
\email{bingronghuangsdu@gmail.com}
\email{rudnick@post.tau.ac.il}

\date{\today}
\begin{abstract}
We  study the distribution of lattice points with prime coordinates lying in the dilate of a convex planar domain having smooth boundary, with nowhere vanishing curvature.  Counting lattice points weighted by a von Mangoldt function gives an asymptotic formula, with the main term being the area of the dilated domain, and our goal is to study the  remainder term. Assuming the Riemann Hypothesis, we give a sharp upper bound, and further assuming that the positive imaginary parts of the zeros of the Riemann zeta functions are linearly independent over the rationals allows us to  give a formula for the value distribution function of the properly normalized remainder term.
 \end{abstract}


\dedicatory{Dedicated to Dorian Goldfeld on the occasion of his 71$^{st}$ birthday}

\maketitle

\section{Introduction}
Our goal in this note is to investigate the distribution of lattice points with prime coordinates lying in dilates of a planar convex set. Before stating our findings, we recall what is known for the classical lattice point problem.

\subsection{Lattice points} There is a vast body of work dedicated to the question of the number of lattice points lying in the family of dilates of a planar domain. A typical context is when one takes a domain $\Omega\subset \R^2$, which is compact, convex, contains the origin in its interior,  with smooth boundary having nowhere zero curvature. One can call such a domain an ``oval". For $R>0$, let $R\Omega$ denote the dilated domain, and let
$$
N_\Omega(R) := \#\Z^2\cap R\Omega
$$
be the number of lattice points in the dilated domain $R\Omega$. Under our assumptions, it is known that
$ N_\Omega(R)\sim \area(\Omega)R^2$ as $R\to \infty$,
and much work has been devoted to bounding the size of the remainder term.
 In the beginning of the 20-th century, it was shown 
that
$$  N_\Omega(R) =\area(\Omega)R^2+O(R^{2/3})$$
and since then  the exponent $2/3$ has been improved somewhat, starting with van der Corput \cite{vdC}, see  \cite{Huxley world record}.
It is conjectured that the correct exponent is $1/2+o(1)$.
 One cannot improve the exponent beyond $1/2$, and there exists arbitrarily large $R$ such that the remainder term is, in absolute value, $\gg R^{1/2}(\log R)^{1/4}$ \cite{Nowak}; in the case of the circle, this is a classical result of Hardy (see \cite{Soundararajan}).
Note that  if one allows points of vanishing curvature, then the remainder term can in some cases be larger than $R^{2/3}$, for instance in the case of the superellipse $\{x^{2k}+y^{2k}\leq  1\}$ the remainder term may be larger than $R^{1-1/(2k)}$ for arbitrarily large $R$ \cite{Randol1, CdV}.

The normalized remainder term
$$ F_\Omega(R):= \frac{ N_\Omega(R) - \area(\Omega)R^2}{R^{1/2}}$$
has a limiting value distribution \cite{Wintner, BCDL, BleherDuke}, that  is there is a measure $\dd\nu_\Omega$ so that for any bounded (piecewise) continuous function $G$,
$$\lim_{T\to \infty} \frac 1T \int_0^T G\Big(F_\Omega(R) \Big) \dd R = \intinf G(u) \dd \nu_\Omega(u).$$
When $\Omega$ is a circle \cite{HB1, BCDL}, or for certain ellipses \cite{Btori}, the limiting distribution   is absolutely continuous, that is $\dd\nu_\Omega(u) = f_\Omega(u)\dd u$; and  the density $f_\Omega$  is real analytic, in particular supported on all of the real line, though with very rapidly decaying tails:
$ f_\Omega(u)\ll \exp(-u^{4 })$ as $|u|\to \infty$,
and in particular is non-Gaussian.


Our goal here is to treat the problem of counting {\em prime} lattice points in the dilated domain $R\Omega$, seeking to address the analogue of the above properties of the lattice point count $N_\Omega(R)$.

\subsection{Prime lattice points}
We consider a convex domain $\Omega$, which we assume is symmetric about the coordinate axes, that is under the reflections $(x,y)\mapsto (\pm x, \pm y)$ (this is natural if we want to count primes, which do not come with a definite sign).
We further assume that $\Omega$ is an {\em oval}, meaning convex with smooth boundary $\partial \Omega$ having nowhere zero curvature. This latter assumption is made as a convenient working hypothesis, and there are interesting variants of the problem which do not satisfy this assumption.

Let
$$\pi_\Omega(R) := \#\Big\{(p,q)\in R\Omega\cap \Z^2: |p|,|q| \quad {\rm prime} \Big\}$$
be the number of lattice points in the dilated domain $R\Omega$ with both coordinates being prime. Also let
$$
\psi_\Omega(R) := \sum_{(m,n)\in R\Omega\cap \Z^2} \Lambda(|m|)\Lambda(|n|)
$$
where $\Lambda(n)$
is  the von Mangoldt function,  which equals $\log p$ if $n=p^k$ is a power of a prime $p$ ($k\geq 1$), and is zero otherwise, and
the sum is over all lattice points lying in the dilated domain $R\Omega$, whose coordinates are both prime powers. To start our investigation, we give a prime number theorem for lattice points in $R\Omega$:
\begin{theorem} \label{PNT for omega}
Assume that $\Omega$ is a symmetric oval as above. Then
$$ \pi_\Omega(R) \sim \area(\Omega) \frac{R^2}{(\log R)^2} \quad {  and} \quad  \psi_\Omega(R) \sim \area(\Omega)R^2,
\quad {  as} \quad R\to \infty .
$$
Assuming the Riemann Hypothesis (RH), we have
$$
\psi_\Omega(R) = \area(\Omega)R^2+O(R^{3/2}) .
$$
\end{theorem}


Our main goal is to study the distribution of the normalized remainder term
$$
H_\Omega(R) = \frac{\psi_\Omega(R)-\area(\Omega)R^2}{R^{3/2}} .
$$
The appropriate scale to use is logarithmic:  We show that assuming RH, there is a probability measure
$\dd\mu_\Omega$, supported in $[-A,A]$ (where $A=\sup |H_\Omega|$) so that for any bounded continuous function $G$,
$$
\lim_{X\to \infty} \frac 1{\log X}\int_1^X G\Big( H_\Omega(R)\Big) \;\frac{\dd R}{R}  = \intinf G(u) \dd\mu_\Omega(u) .
$$

To proceed further, we need the Linear Independence Hypothesis (LI) for the zeros of the Riemann zeta function. Recall that the Riemann Hypothesis is the statement that the nontrivial zeros of the Riemann zeta function are of the form $\rho=\tfrac 12 +i\gamma$, with $\gamma$ real. Due to the functional equation of the Riemann zeta function, the zeros come in conjugate pairs, so that if $\rho=\tfrac 12 +i\gamma$ is a zero, then so is $\bar\rho=\tfrac 12 -i\gamma$. In addition to assuming RH, we assume:
\begin{customthm}{LI}\label{LI}
The imaginary parts of all nontrivial zeros $\rho=\tfrac 12 +i\gamma$  with $\gamma>0$, are   linearly independent   over the rationals.
\end{customthm}
This hypothesis was  used by Wintner \cite{Wintner 1941}, and extensively since, for instance in the study of prime number races \cite{RubSar}. While plausible, it seems unlikely to be provable in the foreseeable future. See \cite[Table 2]{BT} for numerical checks that the first few zeros do not satisfy any linear relations with small coefficients, for instance that the first $500$ zeros do not admit any nontrivial linear relations with coefficients of size at most $10^5$.

From general properties of the value distribution of uniformly almost periodic functions with linearly independent frequencies \cite{WintnerDioph}, we deduce that
\begin{theorem}\label{thm random}
 Assume  the Linear Independence Hypothesis. Then $\dd\mu_\Omega(u) = p_\Omega(u)\dd u$ is absolutely continuous, with a smooth density $p_\Omega$, which is {\em symmetric}: $p_\Omega(-u) = p_\Omega(u)$.
It is the probability distribution function of the  random function
$$
g_\Omega(\vec x) =\sum_{n=1}^\infty A_n \cos(x_n)
$$
where $\vec x=(x_1,x_2,\dots)$ are independent random variables, uniformly distributed in $[0,\pi]$,
$$ A_n=\betaomega(\gamma_n)$$
where $\{\gamma_n:n=1,2,\dots\}$ are the imaginary parts of the nontrivial Riemann zeros ($\gamma_n>0$), and $\betaomega$ is a certain   function depending on the domain (see \eqref{def of BOmega}), satisfying
$$
\betaomega(\gamma)\ll \gamma^{-3/2},\quad \gamma \to \infty,
$$
and is nonzero infinitely often.
\end{theorem}

Theorem~\ref{thm random} allows us to use a formula for the limiting distribution of a sum of sine waves with random phases
  to deduce
\begin{equation}\label{Flowers}
p_\Omega(u) = \frac 1{2A} +\frac 1A\sum_{k=1}^\infty \Big(\prod_{n=1}^\infty J_0\big(\frac{\pi k A_n}A\big) \Big) \cos\Big(\frac{\pi k u}A\Big), \qquad |u|<A
\end{equation}
where $A=\sum_n A_n$.  See Figure~\ref{fig ellipse vs circle} for the value distribution for the circle and for an ellipse, by using \eqref{Flowers} with 500 zeros.

  \begin{figure}[!htb]
\begin{center}
  \includegraphics[width=80mm]
{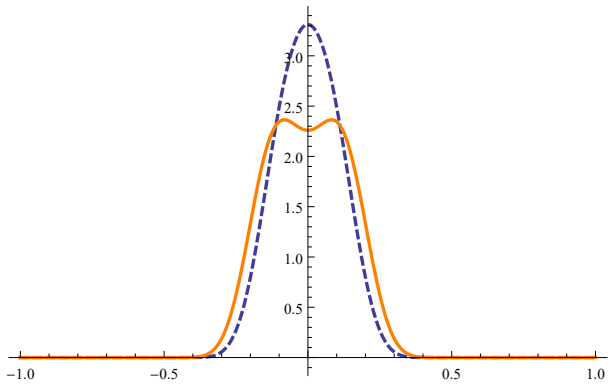}
 \caption{ The value distribution function $p_\Omega(u)$    for the   circle $x^2+y^2\leq 1$ (dashed) and the ellipse $(x/a)^2+(y/b)^2\leq 1$ (solid) with $a=1$, $b=0.65$. The plots have been rescaled, replacing $p_\Omega(u)$ by $Ap_\Omega(Au)$. Note that for this particular ellipse, the distribution is bimodal. }
 \label{fig ellipse vs circle}
\end{center}
\end{figure}

According to Theorem~\ref{thm random},   the value distribution function $p_\Omega$ is {\em symmetric}: $p_\Omega(-u) = p_\Omega(u)$. Note that for the corresponding problem of counting all lattice points,  the distribution need not be symmetric, for instance for the circle, the third moment is negative \cite{Tsang}.

As mentioned earlier, the assumption that $\Omega$ has smooth boundary, with nowhere zero curvature, is made to get a simple set of examples. There are other natural cases one can consider, for instance when $\Omega$ is the   triangle $T  =\{x+y\leq 1, x,y>0\}$. Then $\psi_T(R)=\sum_{n=1}^R r(n)$
where $r(k) = \sum_{m+n=k}\Lambda(m)\Lambda(n)$,   so that $\psi_T(R)/R$
is related to the average number of representations of an integer as a sum of two primes. In this guise, the value distribution of  $(\psi_T(R)-\tfrac 12R^2)/R^{3/2}$  was studied by Fujii \cite{Fujii}.
 See Figure~\ref{overlaid_goldbach} for a plot of the corresponding value distribution function $p_T$, by using \eqref{Flowers} with 500 zeros.
\begin{figure}[!htb]
\begin{center}
  \includegraphics[width=80mm]
{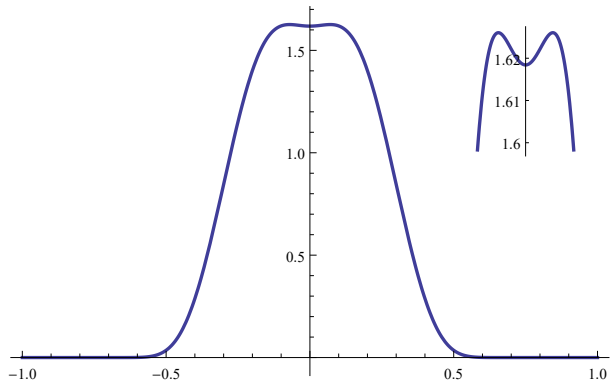}
 \caption{ The value distribution function $p_T$ for the triangle $T=\{x+y\leq 1,x,y\geq 0\}$. The inset displays the bimodal nature of the distribution.}
 \label{overlaid_goldbach}
\end{center}
\end{figure}

\subsection{The one-dimensional case}

It is instructive to compare our findings on the remainder term for ovals in dimension two with the one-dimensional case, where we take a symmetric interval $\Omega=[-1,1]$, and then
$$\psi_{\Omega}(R)=2\psi(R)=2\sum_{n\leq R} \Lambda(n)$$
and we are simply studying the remainder term in the Prime Number Theorem. In that case, Littlewood showed \cite{Littlewood 1914}  (assuming RH) that the normalized remainder term
$(\psi(R)-R)/R^{1/2}$
is {\em unbounded}, unlike what we find in the case of 2-dimensional ovals.   Wintner \cite{Wintner 1941} proved the existence of a limiting distribution $p(u)$ (assuming RH), which is not compactly supported.
In comparison,   for our symmetric ovals, the normalized remainder term  $(\psi_\Omega(R)-\area(\Omega)R^2)/R^{3/2}$ is {\em bounded}, so that the limiting distribution
$p_\Omega$ is compactly supported.

\bigskip

\noindent{\bf Acknowledgements:}
We thank Steve Lester and the reviewer for their comments.
The work was supported  by the European Research Council, under the European Union's Seventh
Framework Programme (FP7/2007-2013)/ERC grant agreement n$^{\text{o}}$~320755.

\section{Symmetric ovals}

\subsection{Geometric preliminaries}
We take a planar domain  $\Omega$ to be an {\em oval}, that is bounded by a smooth, convex curve, which has nowhere-vanishing curvature. We further assume that $\Omega$ is \emph{symmetric} with respect to reflections in the coordinate axes $(x,y)\mapsto (\pm x,\pm y)$, so it necessarily contains the origin.
We may then display the top half of the boundary as the graph of a function:
$$\partial\Omega\cap \{y>0\} = \Big\{ (x,f(x)) : |x|\leq a\Big\}$$
where $f(x)$ is an even function (to take into account the reflection symmetry in the $y$-axis), which is smooth, $f(a)=0$, $f(x)$ is monotonically decreasing for $x>0$ (to allow convexity),  and $f''(x)<0$ to give the nowhere vanishing curvature condition, since the curvature of $\Omega$ at $(x,f(x))$ is
$$\kappa(x,f(x)) = -\frac{f''(x)}{(1+f'(x)^2)^{3/2}}, \qquad |x|<a .
$$
Likewise, we may display the right half of the boundary as a graph:
$$\partial\Omega\cap \{x>0\} = \Big\{ (g(y),y) : |y|\leq b\Big\}$$
with
$$ g=f^{-1}$$
the inverse function to $f$.

For instance, if $\Omega$ is the ellipse $(x/a)^2+(y/b)^2\leq 1$, then we take $f(x)=b\sqrt{1-(x/a)^2}$, $|x|\leq a$, and $g(y) = a\sqrt{1-(y/b)^2}$, $|y|\leq b$.

Other   examples are {\em Cassini ovals}, which are the locus of points such that the product of their distances from two fixed points a distance $2\alpha$ apart is a constant  $\beta^2$. In cartesian coordinates, if we locate the two points on the $x$-axis at $(\pm \alpha,0)$, then the equation of the boundary curve is
\begin{equation}\label{cassini}
 \Big( (x-\alpha)^2+y^2\Big) \Big( (x+\alpha)^2+y^2 \Big)=\beta^4,
\end{equation}
which intersects the $x$-axis at $\pm \sqrt{\beta^2+\alpha^2}$, and the $y$-axis at $\pm\sqrt{\beta^2-\alpha^2}$ (assuming $\beta>\alpha$) .
If $\beta>\sqrt{2}\alpha$ then we get an oval,   if $\alpha<\beta<\sqrt{2}\alpha$ then we get a non-convex curve (a ``dog-bone"),
see Figure~\ref{fig:cassinis}, while for $0<\beta<\alpha$ we get two disconnected curves. For the Cassini oval \eqref{cassini} with $\beta>\sqrt{2}\alpha$, we take
$$f(x) =\sqrt{\sqrt{4 \alpha^2 x^2+\beta^4}-\alpha^2-x^2} \;, \quad |x|\leq \sqrt{\beta^2+\alpha^2} $$
and
$$g(y) = \sqrt{\alpha ^2-y^2+\sqrt{\beta ^4-4 \alpha ^2 y^2}}
\;,\quad |y|\leq \sqrt{\beta^2-\alpha^2}.$$
 \begin{figure}[!htb]
\begin{center}
  \includegraphics[width=80mm]
  {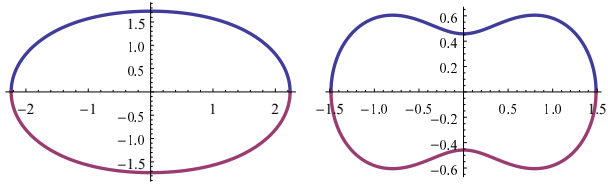}
 \caption{The Cassini ovals $  (  (x-\alpha)^2+y^2 ) ( (x+\alpha)^2+y^2 )=\beta^4  $
with $\alpha=1$, and $\beta=2$ (LHS) and $\beta=1.1$ (RHS).}
 \label{fig:cassinis}
\end{center}
\end{figure}

\subsection{Singularities at the vertices}
 We note that for symmetric ovals, the intersection points  $\{(\pm a,0),(0,\pm b)\}$ of $\partial \Omega$ with the coordinate axes are {\em vertices}, that is local extrema of the curvature.
We will need to know the nature of the singularities of $f(x)$ as $x\nearrow  a$ and of the inverse function $g(y)=f^{-1}(y)$ as $y\nearrow b$:

\begin{lemma}\label{lem:sing}
Let $\kappa(x,y)$ be the curvature at a point $(x,y)\in \partial \Omega$ of the boundary. Then
\begin{equation}\label{sing of f}
f(x) = \sqrt{\frac{2}{\kappa(a,0)} } \cdot  \sqrt{a-x} \cdot \Big( 1+O(a-x) \Big), \quad {\rm as }\; x\nearrow a
\end{equation}
and
\begin{equation}\label{sing of f^{-1}}
g(y) =\sqrt{\frac{2}{\kappa(0,b)} } \cdot  \sqrt{b-y} \cdot \Big(1+O(b-y) \Big), \quad {\rm as }\; y\nearrow b.
\end{equation}
\end{lemma}

\begin{proof}
We write $x=g(y)$ for $y\searrow 0$, when $x\nearrow a$, and expand $g(y)$ in a Taylor series around $y=0$
$$x=g(y) = g(0) + g'(0)y+\frac 12 g''(0)y^2 + \frac 1{3!}g^{(3)}(0)y^3 +O(y^4).$$
We use $g(0)=a$ and   the vanishing of the odd derivatives  at $0$ since $g$ is even: $g'(0)=0= g^{(3)}(0)$, and obtain
$$
x=a + \frac 12 g''(0)y^2 +O(y^4)
$$
or
$$
y=f(x)=\sqrt{\frac{-2}{g''(0)}}\sqrt{a-x}\Big( 1+O(x-a)\Big).
$$
Now we recall that the curvature of a graph $(g(y),y)$ is given by
$$\kappa(g(y),y) = -\frac{g''(y)}{(1+g'(y)^2)^{3/2}}$$
and at $y=0$ this reduces to
\begin{equation}\label{g''=kappa}
\kappa(a,0) = -g''(0).
\end{equation}
Hence we have found
$$
f(x)=\sqrt{\frac{2}{\kappa(a,0)}}\sqrt{a-x}\Big( 1+O(a-x)\Big),\quad x\nearrow a
$$
giving \eqref{sing of f}. The argument for \eqref{sing of f^{-1}}  is identical.
\end{proof}

\begin{lemma}\label{lem:sing2}
Let
$$A(x):=f(x)- \sqrt{\frac{2}{\kappa(a,0)}(a-x) },\qquad B(y):= g(y)-\sqrt{\frac{2}{\kappa(0,b)}(b-y) }   .
$$
Then the derivative of $A$ satisfies  $A'(x)=O(\sqrt{a-x})$ as $x\nearrow a$, and in particular $A'(a)=0$.
The second derivative of $A$ satisfies
$$A''(x) = O\Big(\frac{1}{\sqrt{a-x}}\Big), \quad {\rm as }\; x\nearrow a
$$
and in particular $A''$ is integrable on $(0,a)$. Likewise,  $B'(b)=0$ and $B''$ is   integrable on $(0,b)$.
\end{lemma}
\begin{proof}
From $x=g(y)$ we have $ 1=g'(y)y'$ or
\begin{equation}\label{deriv of g}
 y'=\frac 1{g'(y)}
\end{equation}
Hence the second derivative of $y=f(x)$ is given by
$$ y'' = (\frac 1{g'(y)})' = -\frac {g''(y)y'}{(g'(y))^2} = -\frac {g''(y) }{(g'(y))^3}
$$
after inserting \eqref{deriv of g}.

Expanding about $y=0$, and recalling that since $g$ is even, all the odd derivatives vanish at $y=0$, we obtain
$$ g''(y) = g''(0) + O(y^2)=-\kappa(a,0)\Big(1+O(y^2)\Big)$$
after using \eqref{g''=kappa}, and
$$ g'(y) = 0+g''(0)y+O(y^3) = -\kappa(a,0)y\Big(1+O(y^2)\Big).
$$
Hence
$$
y' = \frac 1{g'(y)}  = -\frac 1{\kappa(a,0)y}\Big(1+O(y^2)\Big) .
$$
Inserting $y=\sqrt{ \frac{2}{\kappa(a,0)}(a-x)}(1+O(a-x))$ we obtain
$$f'(x) = y' =- \frac 1{\sqrt{2\kappa(a,0)(a-x)}} + O\Big(\sqrt{a-x}\Big),  \quad x\nearrow a
$$
and so
$$ A'(x) = f'(x)-\Big( \sqrt{\frac 2{\kappa(a,0)}(a-x)} \Big)'  = O\Big(\sqrt{a-x} \Big)$$
and in particular, $A'(a)=0$.

Similarly
$$
y'' =  -\frac {g''(y) }{(g'(y))^3}= -\frac {- \kappa(a,0)(1+O(y^2))}{( -\kappa(a,0)y)^3(1+O(y^2))} =- \frac 1{\kappa(a,0)^2 y^3} +O\Big(\frac 1{y}\Big) .
$$
and so
$$
f''(x) = -\frac 1{\sqrt{8\kappa(a,0)}}\frac 1{(a-x)^{3/2}} +O\Big(\frac 1{\sqrt{a-x}}\Big)
$$
which gives, after identifying the first term as the second derivative of $\sqrt{\frac{2}{\kappa(a,0)} (a-x)}$, that
$$
A''(x) = f''(x) - \Big(\sqrt{\frac{2}{\kappa(a,0)} (a-x)} \;\Big)'' =  O\Big(\frac 1{\sqrt{a-x}}\Big)
$$
as claimed.
\end{proof}

\subsection{An oscillatory integral}
Given a symmetric oval $\Omega$ as above, define
\begin{equation}\label{def osc}
\oscI_1(\rho):=\int_0^a f(x)x^{\rho-1 }\dd x, \qquad  \oscI_2(\rho):= \int_0^b g(y)y^{\rho-1}\dd y
\end{equation}
which are the Mellin transforms of $f$ and $g$.
We want to asymptotically evaluate the oscillatory integrals as $|\Im \rho|\to +\infty$ ($|\Re\rho|\leq 1/2$). The result is
\begin{lemma}\label{lem:asymp tilde f}
Let  $\kappa(x,y)$ be the curvature of the boundary  $\partial \Omega$ at the point $(x,y)$.   Then
$$
\oscI_1(\rho) =\sqrt{\frac{\pi}{2\kappa(a,0)} }\; \frac{a^{\tfrac 12 +\rho}}{\rho^{3/2}}  +O\Big(\frac 1{|\rho|^2}\Big) ,\quad
\oscI_2(\rho) =  \sqrt{\frac{\pi}{2\kappa(0,b)} }\; \frac{b^{\tfrac 12 +\rho }}{\rho^{3/2}}
 + O\Big(\frac 1{|\rho|^2}\Big)
$$
as $|\rho|\to \infty$ ($|\Re\rho|\leq 1/2$).
\end{lemma}

\begin{proof}
 We use  Lemma~\ref{lem:sing2} to write
$$
f(x) =\sqrt{\frac{2}{\kappa(a,0)} } \cdot  \sqrt{a-x}    +A(x)
$$
with $A''\in L^1(0,a)$, and insert this into the integral $\oscI_1$ to obtain
\begin{equation}\label{first expansion of I}
\oscI_1 (\rho)=\sqrt{\frac{2}{\kappa(a,0)} } \int_0^a \sqrt{a-x} \cdot x^{\rho-1}\dd x
+ \int_0^a A(x) x^{\rho-1}\dd x .
\end{equation}
We have
\begin{equation*}
\begin{split}
\sqrt{\frac{2}{\kappa(a,0)} } \int_0^a \sqrt{a-x} \cdot x^{\rho-1}\dd x &=
\sqrt{\frac{2}{\kappa(a,0)} }\; a^{\tfrac 12+\rho } \int_0^1(1-z)^{1/2}z^{\rho-1} \dd z
 \\
&= \sqrt{\frac{2}{\kappa(a,0)} }\; a^{\tfrac 12 +\rho }B(\tfrac 32, \rho)
\end{split}
\end{equation*}
where $B(x,y)$ is the  Euler Beta function.
By Stirling's formula,
$$
B(\tfrac 32, \rho) =\frac{\frac 12 \sqrt{\pi}}{\rho^{3/2}}\Big(1+O\big(\frac 1{|\rho|}\big)\Big) .
$$

For the second term in \eqref{first expansion of I}, we can integrate by parts twice, using $A(a)=A'(a)=0$, to find
$$\int_0^a A(x) x^{\rho-1}\dd x= \frac 1{\rho(\rho+1)} \int_0^a A''(x)x^{\rho+1} dx=O\Big(\frac 1{|\rho|^2}\Big)
$$
since $A''$ is integrable by Lemma~\ref{lem:sing2}. Thus we find
$$
\oscI_1(\rho)=\sqrt{\frac{\pi}{2\kappa(a,0)} }\; \frac{a^{\tfrac 12 +\rho}}{\rho^{3/2}}
+O\Big(\frac 1{|\rho|^2}\Big)
$$
as claimed. The integral $\oscI_2$ can be  treated identically.
\end{proof}

Lemma~\ref{lem:asymp tilde f} gives an upper bound for $\oscI_1(\rho) + \oscI_2(\rho)$. In \S \ref{sec:distribution}, we will also need a non-vanishing result for $\oscI_1(\rho) + \oscI_2(\rho)$. The following lemma will suffice:

\begin{lemma}\label{lemma:I>>}
  For infinitely many (in fact a positive proportion) of the zeros $\rho = \tfrac 12+i\gamma$, we have
  $$ |\oscI_1(\rho) + \oscI_2(\rho)| \gg \frac 1{\gamma^{3/2}}$$
  and in particular $\oscI_1(\rho) + \oscI_2(\rho)$  is nonzero infinitely often.
\end{lemma}

\begin{proof}
According to Lemma~\ref{lem:asymp tilde f},
we have
\[
\begin{split}
   |\oscI_1(\rho) + \oscI_2(\rho)| 
  &=
  \frac{\sqrt{\pi}}{\sqrt{2}|\rho|^{3/2}} \; \Big|  \frac{a}{\sqrt{\kappa(a,0)}}a^{i\gamma} + \frac{b}{\sqrt{\kappa(0,b)}} b^{i\gamma} \Big|
  +O\Big(\frac 1{|\rho|^2}\Big)
\\
&= \frac {C}{ |\rho|^{3/2}}\Big| e^{i\gamma \log(b/a)}+c\Big|+O\Big(\frac 1{|\rho|^2}\Big)
\end{split}
\]
where $C>0$, $c>0$ are independent of $\gamma$.

Now if $a=b$ then we get
$$ |\oscI_1(\rho) + \oscI_2(\rho)| \sim \frac{C(1+c)}{\gamma^{3/2}}\gg \frac 1{\gamma^{3/2}}$$
for all $\gamma\gg 1$ as we claim.

If $a\neq b$, that is $\log(b/a)\neq 0$, then we use a result of Hlawka \cite{Hlawka} (see also Rademacher \cite{Rademacher} for a proof assuming RH), for any $\alpha\neq 0$, the sequence $\{\alpha\gamma :\gamma>0\}$ is uniformly distributed modulo one. Hence for a positive proportion of $\gamma$'s, we have $\Re e^{i\gamma \log(b/a)}>\tfrac 12$, and hence for these we have
$$ \Big| e^{i\gamma \log(b/a)}+c\Big| \geq \frac 12 +c\geq \frac 12$$
so that for a positive proportion of $\gamma$'s,
$$
 |\oscI_1(\rho) + \oscI_2(\rho)| \gg \frac 1{\gamma^{3/2}}
$$
as claimed.
\end{proof}

\section{Counting prime points}\label{sec:PP_f}

Now we want to consider prime points in a symmetric oval $\Omega$ as above.

Let $y=f(x)$ be the function which gives the boundary of $\Omega$ in the  first quadrant. By our assumptions, we know $f(x)$ satisfies that
\[
  f(0)=b>0,\ f(a)=0, \
   f'(0)=0 \quad \textrm{and} \quad
   f'(x)\searrow-\infty \ \textrm{as}\ x\nearrow a.
\]
Since the curvature of the boundary is non-vanishing, we know that $f'(x)<0$ for all $x\in(0,a)$.
Let $R>0$ be a large parameter.

\subsection{The main term}
We first give  the main terms    in Theorem~\ref{PNT for omega}, as a simple consequence of the Prime Number Theorem:
\begin{proposition}
Let $\Omega $ be a symmetric oval. Then
$$\psi_\Omega(R)\sim \area(\Omega)R^2 \quad and \quad \pi_\Omega(R)\sim  \frac{\area(\Omega)R^2}{(\log R)^2}   ,\qquad R\to \infty .$$
\end{proposition}


\begin{proof}
Using the symmetry of $\Omega$, it suffices to perform the analysis in the positive quadrant, where we sum over lattice points with prime power coordinates lying under the graph of $y=f(x)$:
\begin{equation*}
      \psi_{\Omega}(R)
     =4 \sum_{m\leq aR} \Lambda(m) \sum_{n\leq Rf(m/R)}\Lambda(n) .
\end{equation*}
By the Prime Number Theorem, the inner sum is
\[
\sum_{n\leq Rf(m/R)}\Lambda(n)  =Rf(\frac {m}{R}) +o(R)
\]
and so
\[
 \psi_{\Omega}(R)= 4R \sum_{m\leq aR}\Lambda(m)f(\frac {m}{R}) +o(R^2).
\]
Applying summation by parts, using the Prime Number Theorem again,    gives
\[
 \psi_{\Omega}(R)= 4R^2\int_0^a f(v)\dd v +o(R^2) = \area(\Omega)R^2 + o(R^2).
\]

To prove the claim about $\pi_\Omega$, we first bound the contribution to $\psi_\Omega(R)$
of pairs $(m,n)$ where at least one of them is less than $R/(\log R)^{10}$ by
\begin{multline*}
 \sum_{\substack{m ,n\ll R\\ \min(m,n)<R/(\log R)^{10}}}\Lambda(m) \Lambda(n)
\\
\ll  (\log R)^2 \#\{(m,n): m,n\ll R,\min(m,n)<\frac{R}{(\log R)^{10}} \}
\\
 \ll (\log R)^2 \frac{R^2}{(\log R)^{10}}
\end{multline*}
which is negligible for our purposes.

Moreover, the contribution of $(m,n)$ for which at least one is not a prime, is bounded by
$$
\ll \log R\sum_{p\ll R^{1/2} } \log p \sum_{q\ll R} \log q\ll R^{3/2}(\log R)^3
$$
which is again negligible. Thus
$$
\frac 14\psi_\Omega(R)\sim \sum_{R/(\log R)^{10}<p<aR}\log p\sum_{R/(\log R)^{10}<q<Rf(p/R)}\log q
$$
the sum over primes.

For $p\in (R/(\log R)^{10},R)$, we have $\log p\sim \log R$ and likewise for the sum over  $q$. Hence we find
$$
\frac 14\psi_\Omega(R)\sim (\log R)^2  \sum_{R/(\log R)^{10}<p<aR} \sum_{R/(\log R)^{10}<q<Rf(p/R)}1.
$$
Arguing as above, we find
$$
\sum_{R/(\log R)^{10}<p<aR} \sum_{R/(\log R)^{10}<q<Rf(p/R)}1 =\frac 14 \pi_\Omega(R) + O\Big(\frac{R^{2}}{(\log R)^{10}}\Big).
$$
Therefore we find
$$\psi_\Omega(R)\sim (\log R)^2 \pi_\Omega(R)$$
and hence
$$
\pi_\Omega(R)\sim  \frac{\area(\Omega)R^2}{(\log R)^2}
$$
as claimed.
\end{proof}


\subsection{Using RH}

In this section, we give a formula for $\psi_\Omega(R)$ in terms of a sum over zeros of the Riemann zeta function: Define
\begin{equation}\label{eqn:H_f}
  \tilde H_{\Omega}(R) := -4 \sum_{\rho} R^{\rho-1/2} \Big(\oscI_1(\rho) + \oscI_2(\rho) \Big)
\end{equation}
where the Mellin transforms $\oscI_j$ are given in \eqref{def osc}.
Then we show that up to a negligible error, $\tilde H_\Omega(R)$ coincides with the normalized remainder term $H_\Omega(R) = (\psi_\Omega(R)-\area(\Omega)R^2)/R^{3/2}$:
\begin{proposition}\label{prop:EF}
Assume RH. Then
\begin{equation*}\label{eqn:PP}
  \psi_{\Omega}(R)
      = \mathrm{area}(\Omega) R^2
       + R^{3/2} \tilde H_\Omega(R)
     + O\left( R^{4/3} (\log R)^{7/2} \right).
\end{equation*}
\end{proposition}

\begin{proof}

By the approximate explicit formula (see e.g. Davenport \cite[\S17, eq. (9) and (10)]{davenport2000multiplicative}),   for $x\geq2$ and $T>1$,
\[
  \sum_{n\leq x}\Lambda(n) = x - \sum_{\substack{\rho\\|\gamma|\leq T}}\frac{x^\rho}{\rho} + O\left(\frac{x(\log xT)^2}{T} + \log x\right) .
\]
Together with the symmetry of $\Omega$, we have
\begin{equation}\label{eqn:applyEF1}
  \begin{split}
     \frac{1}{4} \psi_{\Omega}(R)
     & = \sum_{m\leq aR} \Lambda(m) \sum_{n\leq Rf(m/R)}\Lambda(n) \\
       &  = R\sum_{m\leq aR} \Lambda(m) f(m/R)  \\
       & \hskip 30pt
       -
       \sum_{m\leq aR} \Lambda(m) \sum_{\substack{\rho\\|\gamma|\leq T}}\frac{(Rf(m/R))^{\rho}}{\rho}
       + O\left(\frac{R^2(\log R)^2}{T}\right) \\
       &  =: \textrm{I} + \textrm{II} + O\left(\frac{R^2(\log R)^2}{T}\right),
  \end{split}
\end{equation}
say, where we assume $1<T\ll R$.

By the partial summation, for $1<T'\ll R$, we have 
\[
  \begin{split}
    &\sum_{m\leq aR} \Lambda(m) f(m/R)^\rho
    = \int_{2}^{aR} f(u/R)^\rho \dd \sum_{m\leq u}\Lambda(m) \\
    & \hskip 30pt = \int_{2}^{aR} f(u/R)^\rho \dd
    \bigg(u - \sum_{\substack{\rho'\\|\gamma|\leq T'}}\frac{u^{\rho'}}{\rho'}\bigg)
    + O\left( \frac{R(\log R)^2}{T'} \int_{2}^{aR}   |\dd f(u/R)^\rho| \right) \\
    & \hskip 30pt = \int_{2}^{aR} f(u/R)^\rho \dd
    \bigg(u - \sum_{\substack{\rho'\\|\gamma|\leq T'}}\frac{u^{\rho'}}{\rho'}\bigg)  \\
    & \hskip 90pt
    + O\left(\frac{ R (\log R)^2}{T'} \frac{|\rho|}{R} \int_{2}^{aR} f\Big(\frac{u}{R}\Big)^{-1/2} \Big|f'\Big(\frac{u}{R}\Big)\Big| \dd u \right) \\
    & \hskip 30pt = \int_{0}^{aR} f(u/R)^\rho  \dd u
    - \sum_{\substack{\rho'\\|\gamma'|\leq T'}} \frac{1}{\rho'}
    \int_{2}^{aR} f(u/R)^\rho \dd u^{\rho'}
    + O\left(\frac{|\rho| R (\log R)^2}{T'}\right).
  \end{split}
\]
We have again
\begin{equation*}
  \begin{split}
      \sum_{\substack{\rho'\\|\gamma'|\leq T'}} \frac{1}{\rho'}
    \int_{0}^{2} f(u/R)^\rho \dd u^{\rho'}
    & =  \sum_{\substack{\rho'\\|\gamma'|\leq T'}} \frac{1}{\rho'} \Big(
     f(u/R)^\rho u^{\rho'}\Big|_{0}^{2} -
      \int_{0}^{2} u^{\rho'} \dd f(u/R)^\rho \Big) \\
    & = O \left( (\log R)^2 + |\rho| R^{-1} (\log R)^2 \right).
  \end{split}
\end{equation*}
Hence
\begin{equation} \label{eqn:w-sum-Lambda}
  \begin{split}
    & \sum_{m\leq aR} \Lambda(m) f(m/R)^\rho
    = R\int_{0}^{a} f(v)^\rho  \dd u \\
    & \hskip 60pt - R^{\rho'}\sum_{\substack{\rho'\\|\gamma'|\leq T'}}
    \int_{0}^{a} f(v)^\rho v^{\rho'-1} \dd v
    + O\left(\frac{|\rho| R(\log R)^2}{T'}\right).
  \end{split}
\end{equation}
The above asymptotic formula holds for $\rho=1$ by a similar argument.

Note that $R^2\int_0^a f(v)dv = \tfrac 14\area(\Omega)R^2$.
By \eqref{eqn:w-sum-Lambda} with $\rho=1$, we get
\begin{equation}\label{eqn:firstterm}
\textrm{I}   = \frac 14\area(\Omega)R^2- R\sum_{\substack{\rho\\|\gamma|\leq T}} R^\rho
   \oscI_1(\rho)
  + O\left(\frac{R^2(\log R)^2}{T} \right).
\end{equation}
with the Mellin transform $\oscI_1(\rho)$ given by \eqref{def osc}.

Now we handle the second term. By \eqref{eqn:w-sum-Lambda} again, we have
\begin{equation*}
  \begin{split}
    \textrm{II} & = -\sum_{m\leq aR} \Lambda(m)
    \sum_{\substack{\rho\\|\gamma|\leq T}}\frac{R^\rho f(u/R)^\rho}{\rho}
    = - \sum_{\substack{\rho\\|\gamma|\leq T}} \frac{R^\rho}{\rho}
    \sum_{m\leq aR} \Lambda(m) f(u/R)^\rho \\
    & = - R \sum_{\substack{\rho\\|\gamma|\leq T}} \frac{R^\rho}{\rho}
    \int_{0}^{a} f(v)^\rho \dd v
    + \sum_{\substack{\rho\\|\gamma|\leq T}}
    \sum_{\substack{\rho'\\|\gamma'|\leq T'}} \frac{R^{\rho+\rho'}}{\rho}
    \int_{0}^{a} f(v)^\rho v^{\rho'-1} \dd v
      \\ & \hskip 45pt  + O\left(\frac{T R^{3/2}(\log R)^3}{T'}\right).
  \end{split}
\end{equation*}
We change variable $u=f(v)$, so $v=g(u)$, to transform
\begin{multline*}
  \frac{1}{\rho}  \int_{0}^{a} f(v)^\rho \dd v
  = -\int_0^b \frac{u^\rho}{\rho} \dd g(u)
  = -\frac{u^\rho}{\rho} g(u)\Big|_0^b + \int_{0}^{b} u^{\rho-1} g(u) \dd u
\\ =\int_{0}^{b} u^{\rho-1} g(u) \dd u=:\oscI_2(\rho)
\end{multline*}
and obtain
\begin{equation}\label{eqn:secondterm}
  \begin{split}
    & \textrm{II} = - R \sum_{\substack{\rho\\|\gamma|\leq T}}  R^\rho \oscI_2(\rho)
    + \sum_{\substack{\rho\\|\gamma|\leq T}}
    \sum_{\substack{\rho'\\|\gamma'|\leq T'}} \frac{R^{\rho+\rho'}}{\rho}
    \int_{0}^{a} f(v)^\rho v^{\rho'-1} \dd v \\
    & \hskip 150pt
   + O\Big(\frac{T R^{3/2}(\log R)^3}{T'}\Big).
   \end{split}
\end{equation}

Combining \eqref{eqn:applyEF1}, \eqref{eqn:firstterm}, and \eqref{eqn:secondterm}, and assuming that $2< T \leq T' \ll R$, we have
\begin{equation}\label{eqn:primesindomains}
  \begin{split}
       \psi_{\Omega}(R)
      &=  \area(\Omega)R^2
       - 4R \sum_{\substack{\rho\\|\gamma|\leq T}} R^\rho \Big(\oscI_1(\rho) + \oscI_2(\rho) \Big)
  \\  &\hskip 60pt  + 4 \mathcal S
     + O\left(\frac{R^2(\log R)^2}{T} + \frac{T R^{3/2}(\log R)^3}{T'}\right),
  \end{split}
\end{equation}
where
\begin{equation}\label{eqn:Sigma}
  \mathcal S := \sum_{\substack{\rho\\|\gamma|\leq T}}
    \sum_{\substack{\rho'\\|\gamma'|\leq T'}} \frac{R^{\rho+\rho'}}{\rho}
    \int_{0}^{a} f(v)^\rho v^{\rho'-1} \dd v.
\end{equation}
Then we have
\begin{equation*}
  \begin{split}
    \mathcal S
    & = \sum_{\substack{\rho\\|\gamma|\leq T}}
    \sum_{\substack{\rho'\\|\gamma'|\leq T'}} \frac{R^{\rho+\rho'}}{\rho\rho'} \Big( f(v)^\rho v^{\rho'}\Big|_{0}^{a} -
    \int_{0}^{a} \rho f(v)^{\rho-1} f'(v) v^{\rho'} \dd v \Big)   \\
    & = -\sum_{\substack{\rho\\|\gamma|\leq T}}
    \sum_{\substack{\rho'\\|\gamma'|\leq T'}} \frac{R^{\rho+\rho'}}{\rho'}
    \int_{0}^{a} f(v)^{\rho-1} f'(v) v^{\rho'} \dd v   \\
    & = - \int_{0}^{a} f'(v) \sum_{\substack{\rho\\|\gamma|\leq T}} R^{\rho} f(v)^{\rho-1}
    \sum_{\substack{\rho'\\|\gamma'|\leq T'}} \frac{R^{\rho'}}{\rho'}
    v^{\rho'} \dd v.
  \end{split}
\end{equation*}

We now assume RH, and write the zeros  as $\rho=\tfrac 12 +i\gamma$, $\rho'=\tfrac 12 +i\gamma'$ .
By Cauchy--Schwarz,
\begin{equation*}
  \begin{split}
    \mathcal S & \ll
     R \Big(\int_{0}^{a} |f'(v)| \Big|\sum_{\substack{\rho\\|\gamma|\leq T}} R^{i\gamma} f(v)^{\rho-3/4} \Big|^2 \dd v \Big)^{1/2} \\
     & \hskip 90pt \cdot
    \Big(\int_{0}^{a} |f'(v)| f(v)^{-1/2}
    \Big|\sum_{\substack{\rho'\\|\gamma'|\leq T'}} \frac{R^{i\gamma'}}{\rho'}
    v^{1/2+i\gamma'} \Big|^2 \dd v  \Big)^{1/2} \\
    & \ll  R \Big(\int_{0}^{a} |f'(v)| \Big|\sum_{\substack{\rho\\|\gamma|\leq T}} R^{i\gamma} f(v)^{-1/4+i\gamma} \Big|^2 \dd v \Big)^{1/2} \\
     & \hskip 90pt \cdot
      \Big( (\log R)^4 \int_{0}^{a} |f'(v)| f(v)^{-1/2}\dd v  \Big)^{1/2}
  \end{split}
\end{equation*}
on using, for $|v|\leq a$,
\[
\sum_{\substack{\rho'\\|\gamma'|\leq T'}} \Big| \frac{R^{i\gamma'}}{\rho'}    v^{1/2+i\gamma'} \Big|  \ll
\sum_{\substack{\rho'\\|\gamma'|\leq T'}} \frac 1{|\rho'|} \ll
(\log T')^2 \ll  (\log R)^2 .
\]
Note that $\int_{0}^{a} |f'(v)| f(v)^{-1/2}\dd v = -2f(v)^{1/2}\big|_{0}^{a} = 2\sqrt{b}$. So we have
\[
  \begin{split}
    \mathcal S & \ll R (\log R)^2 \Big(\int_{0}^{a} |f'(v)| f(v)^{-1/2}
    \sum_{\substack{\rho\\|\gamma|\leq T}} \sum_{\substack{\rho'\\|\gamma'|\leq T}} (Rf(v))^{i(\gamma-\gamma')}
    \dd v  \Big)^{1/2} \\
    & \ll R (\log R)^2 \Big(
    \sum_{\substack{\rho\\|\gamma|\leq T}} \sum_{\substack{\rho'\\|\gamma'|\leq T}} R^{i(\gamma-\gamma')}
    \int_{0}^{a} f(v)^{-1/2+i(\gamma-\gamma')} (-f'(v)) \dd v  \Big)^{1/2} .
  \end{split}
\]
Hence we obtain
\[
  \begin{split}
    \mathcal S
    & \ll R (\log R)^2 \Big(
    \sum_{\substack{\rho\\|\gamma|\leq T}} \sum_{\substack{\rho'\\|\gamma'|\leq T}}
    \big|\int_{0}^{a} f(v)^{-1/2+i(\gamma-\gamma')} \dd f(v)\big|  \Big)^{1/2} \\
    & \ll R (\log R)^2 \Big(
    \sum_{\substack{\rho\\|\gamma|\leq T}} \sum_{\substack{\rho'\\|\gamma'|\leq T}}
    \frac{1}{1+|\gamma-\gamma'|}  \Big)^{1/2}
    \ll RT^{1/2} (\log R)^{7/2}.
  \end{split}
\]
Now by taking $T=R^{2/3}$ and $T'=R^{5/6}$, we have
\begin{equation*}\label{eqn:PPinD}
        \psi_{\Omega}(R)
       =  \area(\Omega)R^2
       - 4R \sum_{\substack{\rho\\|\gamma|\leq R^{2/3}}} R^\rho
          \Big(\oscI_1(\rho) + \oscI_2(\rho) \Big)
+ O\left( R^{4/3} (\log R)^{7/2} \right).
\end{equation*}

Using Lemma~\ref{lem:asymp tilde f}, we may extend the sum over all zeros, introducing an error of $O(R^{7/6+o(1)})$ which is negligible relative to the other remainders. Thus define
$\tilde{H}_\Omega(R)$ as in \eqref{eqn:H_f}.
Then we have
\begin{equation*}\label{eqn:PP}
  \psi_{\Omega}(R)
      = \mathrm{area}(\Omega) R^2
       + R^{3/2} \tilde H_\Omega(R)
     + O\left( R^{4/3} (\log R)^{7/2} \right).
\end{equation*}
This completes the proof.
\end{proof}

\begin{corollary}
Assuming RH, we have
$$\psi_{\Omega}(R) =  \mathrm{area}(\Omega) R^2 +O(R^{3/2}).$$
\end{corollary}
\begin{proof}
Indeed, from RH, $|R^{\rho-1/2}|=1$ and inserting that into the definition \eqref{eqn:H_f} of $\tilde H_\Omega(R)$ and using Lemma~\ref{lem:asymp tilde f} shows that $\tilde H_\Omega(R)=O(1)$. The statement then follows from Proposition~\ref{prop:EF}.
\end{proof}


\section{The value distribution function $p_\Omega$}\label{sec:distribution}

We  now compare the empirical remainder term
$$H_\Omega(R) = \frac{\psi_\Omega(R)-\area(\Omega)R^2}{R^{3/2}}$$
with the sum \eqref{eqn:H_f}
$$\tilde H_\Omega(R) = -4\sum_\rho R^{\rho-\tfrac 12} \Big(\oscI_1(\rho)+\oscI_2(\rho)\Big)
$$
the sum over the nontrivial zeros of the Riemann zeta function. Assuming the Riemann Hypothesis, we write them as $\rho=\tfrac 12+i\gamma$, $\gamma\in \R$. As an immediate consequence of Proposition~\ref{prop:EF} we obtain
\begin{lemma}
 The (logarithmic) value distributions of $H_\Omega$ and of $\tilde H_\Omega$ coincide.
\end{lemma}

Therefore, the logarithmic value distribution of $H_\Omega$ is the (ordinary) value distribution of the sum
$$h_\Omega(t) = \tilde H_\Omega(e^t)= -4\sum_\rho   \Big(\oscI_1(\rho)+\oscI_2(\rho)\Big) e^{it\gamma}.$$

Noting that
$$\overline{\oscI_j(\rho)} = \oscI_j(\bar\rho)$$
we find that
\begin{equation}\label{sum for h}
  h_\Omega(t)=-4\sum_\rho   \Big(\oscI_1(\rho)+\oscI_2(\rho)\Big) e^{it\gamma}
=\sum_{\gamma>0} \mathcal B_\Omega(\gamma) \cos(t\gamma+\varphi_\gamma)
\end{equation}
where the sum is over zeros with positive imaginary part $\gamma>0$, and
\begin{equation}\label{def of BOmega}
\mathcal B_\Omega(\gamma) := 8\Big|\oscI_1(\rho)+\oscI_2(\rho)\Big|,\qquad
\varphi_\gamma :=\arg \Big(-\oscI_1(\rho)-\oscI_2(\rho)\Big) .
\end{equation}

According to Lemma~\ref{lem:asymp tilde f},
$$ \mathcal B_\Omega(\gamma)\ll \gamma^{-3/2}, \qquad \gamma\to +\infty.
$$
Since  the $n$-th zero $\gamma_n\approx n/\log n$ by the Riemann--von Mangoldt formula, we see that the sum \eqref{sum for h} is absolutely convergent, and defines a uniformly almost periodic function, hence has a limiting value distribution measure, (see e.g. \cite[Theorem 4.1]{BleherDuke}), as follows from applying the Kronecker--Weyl ergodic theorem:
$$\lim_{T\to \infty} \frac 1T \int_0^T G\Big(h_\Omega(t)\Big) \dd t = \int_{-A}^A G(u)\dd\mu_\Omega(u)$$
for all bounded continuous functions on $[-A,A]$, where
$$A=\max_t |h_\Omega(t)|.$$
Note that by Lemma \ref{lemma:I>>}, we know that $\mathcal B_\Omega(\gamma)$  is nonzero infinitely often.

We now assume the Linear Independence Hypothesis.
Wintner \cite{WintnerDioph} studied the value distribution of a sum of infinitely many cosine waves with incommensurate frequences
$$
h(t):=\sum_{n=1}^\infty a_n\cos(\gamma_n t-\varphi_n)
$$
where $a_n>0$, with $A:=\sum_n a_n<\infty$, and $\{\gamma_n\}$ are linearly independent over the rationals,
showing that there is a smooth\footnote{Smoothness breaks down if we only take a finite sum.} value distribution function $p(u)$, whose characteristic function is given by
\begin{equation}\label{infinite product of J}
\int_{-A}^{A} p(u)e^{isu}\dd u = \prod_{n=1}^\infty J_0(a_ns)
\end{equation}
and that the value distribution is even: $p(u) = p(-u)$.

A form of Nyquist's Sampling Theorem gives a formula for the probability distribution function $p(u)$ that is useful for computational purposes, compare \cite[equation (25)]{Bennett}.\footnote{There is an unfortunate typo in \cite[equation (25)]{Bennett}.}
\begin{lemma}
Let
$$h(t)=\sum_{n=1}^\infty a_n\cos(\gamma_n t-\varphi_n)
$$
where $a_n>0$, with $A:=\sum_n a_n<\infty$, and $\{\gamma_n\}$ are linearly independent over the rationals. Then the value distribution function $p(u)$ of $h$ is smooth and even, and given for $|u|\leq A$ by the convergent Fourier series
\begin{equation}\label{Bennett's formula}
p(u) = \frac 1{2A}
+\frac 1A\sum_{k=1}^\infty   \Big( \prod_{n=1}^\infty J_0\big(\frac{\pi k a_n}{A}\big) \Big)\cos\Big(\frac{\pi k u}{A}\Big), \quad |u|\leq A
\end{equation}
and $p(u)=0$ outside the interval $[-A,A]$.
\end{lemma}
\begin{proof}

We   define a new, $2A$-periodic function on whole real line by
$$
p_{\rm per}(u):=\sum_{\ell\in \Z} p(u+2A\ell)
$$
which is still smooth, and coincides with $p(u)$ on $[-A,A]$.
The Fourier coefficients of $p_{\rm per}(u)$ are
$$
\widehat p_{\rm per}(k) = \frac 1{2A} \int_{-A}^{A} p_{\rm per}(u) e^{-2\pi i k\frac{u}{2A}}  \dd u
= \frac 1{2A}\int_{-\infty}^\infty p(u) e^{-2\pi i k\frac{u}{2A}}  \dd u =\frac 1{2A}\widehat p\Big(\frac{k}{2A}\Big)
$$
 where $\widehat p(\xi) = \int_{-\infty}^\infty p(u)e^{-2\pi i u\xi}\dd u$ is the Fourier transform of $p$.
In particular $\widehat p_{\rm per}(0)=1/(2A)$.
Thus we have
$$
p_{\rm per}(u) = \frac 1{2A}+\sum_{k\neq 0} \frac 1{2A}\widehat p\Big(\frac{k}{2A}\Big) e^{i\pi k\frac{u}{A}}.
$$
The expansion converges pointwise because $p_{\rm per}(u) $ is smooth.
Since $p(u)$ is even, we rewrite
$$
p_{\rm per}(u) = \frac 1{2A}+\frac 1A\sum_{k=1}^\infty \widehat p\Big(\frac{k}{2A}\Big) \cos\Big(\pi k\frac{u}{A}\Big) .
$$
Due to \eqref{infinite product of J}, we have $\widehat p(s) =\prod_{n=1}^\infty J_0(2\pi a_ns)$. Therefore
$$
p_{\rm per}(u) =
\frac 1{2A}+\frac 1A\sum_{k=1}^\infty\Big( \prod_{n=1}^\infty J_0\big(\frac{\pi ka_n}{A}\big) \Big)\cos\Big(\frac{\pi k u}{A}\Big) .
 $$
 Now for $|u|<A$, $p_{\rm per}(u)=p(u)$ and so we obtain the result.
\end{proof}
\begin{corollary}
Assume RH and Hypothesis LI. Then the logarithmic value distribution function $p_\Omega(u)$of $H_\Omega(R)$ is given by \eqref{Bennett's formula} with $\gamma_n>0$ being the imaginary parts of the Riemann zeros, and $a_n=|\mathcal B_\Omega(\gamma_n)|$.
\end{corollary}

 Note that $p_\Omega(u)$ is the probability distribution function (PDF) of the random cosine sum
$$g_\Omega(t) = \sum_{n=1}^\infty \mathcal B_\Omega(\gamma_n)\cos x_n
$$
with $x_n\in [0,\pi]$ uniform independent identically distributed (IID) random variables.
Random combinations of cosine waves of the form
  $$ Z=\sum_{n=1}^N a_n \cos x_n$$
  with $a_n>0$ and $x_n$ uniform IID random variables, have been studied, starting with Lord Rayleigh \cite{Rayleigh}
   in the context of random flights (Pearson's problem of the random walk), where one wants to find the distribution of the sum of $N$ vectors with specified lengths $a_n$ and randomly distributed phases,
$Z$ being the real part of the sum of the random vectors $a_ne^{ix_n}$. They were used for  in the theory of multi-channel carrier telephony (see \cite{Bennett, Slack}) or for modeling sea waves (see \cite{Flower}).

\section{Prime points in a superellipse}\label{sec:PP_ell}

The superellipse is the planar domain $\Omega_k$ bounded by the Lam\'{e} curve   $ x^{2k}+y^{2k}=1$, where $k\geq 2$ is an integer.  The boundary curve $\partial \Omega_k $ is smooth, but the curvature vanishes at the points $(\pm 1,0)$, $(0,\pm 1)$, which are the vertices of the curve, see Figure~\ref{fig superellipse}.
  \begin{figure}[!htb]
\begin{center}
  \includegraphics[width=80mm]
{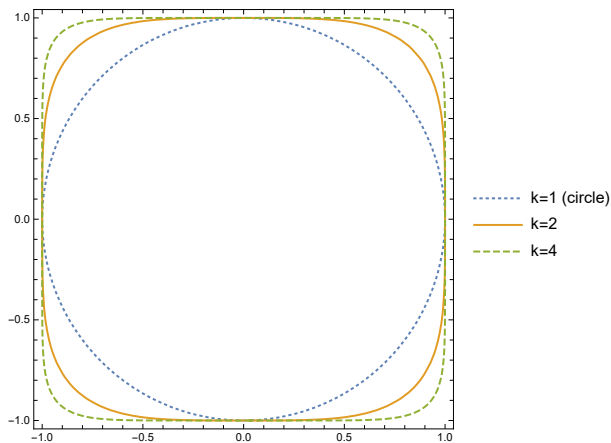}
 \caption{ The superellipse $x^{2k}+y^{2k}\leq 1$, for $k=1$ (a circle), $k=2$ and $k=4$. }
 \label{fig superellipse}
\end{center}
\end{figure}
As noted in the Introduction, the ordinary lattice point count $N_{\Omega_k}(R) = \#\{  \Z^2\cap R\Omega_k\}$ is anomalous in that the remainder term is larger than in ovals, due to the existence of points (namely the vertices) on the boundary $\partial \Omega_k$ where the curvature vanishes to order $2k-2$ (recall $k\geq 2$), and at which the normal to the curve has rational slope. Indeed, already van der Corput in his thesis (see also \cite{Randol1, CdV}) showed that the remainder term $N_{\Omega_k}(R)-\area(\Omega_k)R^2$ is as large as $R^{1-1/(2k)}$ for arbitrarily large $R$, unlike the upper bound of $O(R^{2/3})$ (and conjecturally $O(R^{1/2+o(1)})$) for ovals.

We examine our prime lattice point count $\psi_{\Omega_k}(R)$ for the superellipse, and find that unlike the ordinary lattice point count, the prime lattice point count behaves in the same way as it does for ovals, namely  that
$$ \psi_{\Omega_k}(R) =\area(\Omega_k) R^2+O(R^{3/2})$$
and that the remainder term
$$ H_{\Omega_k}(R) = \frac{ \psi_{\Omega_k}(R) -\area(\Omega_k)R^2}{R^{3/2}}
$$
has a limiting (logarithmic) distribution function, given by a similar formula as for the case of ovals.

\bigskip

Most arguments in \S\ref{sec:PP_f}  carry over to this case.
The only change is in the asymptotic evaluation of the Mellin transforms \eqref{def osc} in Lemma~\ref{lem:asymp tilde f}, where nonvanishing curvature at the vertices is used.
Here, we can evaluate them  directly:
The boundary in the positive quadrant is defined as the graph of the function $f(x) = (1-x^{2k})^{1/(2k)}$ which coincides with its inverse: $g(y)=f(y)$. The Mellin transforms are given by
$$
\oscI_1(\rho) = \oscI_2(\rho) = \int_0^1 (1-x^{2k})^{1/(2k)}x^{\rho-1}dx
= \frac 1{2k} B\Big(1+\frac 1{2k},\frac{\rho}{2k}\Big)
$$
and hence
$$
\oscI_1(\rho)\sim \frac{ \Gamma(1+\tfrac 1{2k})}{2k}\rho^{-(1+\tfrac 1{2k})}
$$
by Stirling's formula (note the exponent $1+\tfrac 1{2k}$ is smaller than the exponent $3/2$ obtained in Lemma~\ref{lem:asymp tilde f} for ovals).
Hence we obtain
$$
\psi_{\Omega_k}(R) = \area(\Omega_k)R^2 + R^{3/2}\tilde H_{\Omega_k}(R) + O\Big( R^{4/3}(\log R)^{7/2} \Big)
$$
with
$$
\tilde H_{\Omega_k}(R) = -\frac {4}{k}\sum_\rho B\Big(1+\frac 1{2k},\frac{\rho}{2k}\Big)  R^{\rho-1/2} .
$$
See Figure~\ref{fig superellipsedist} for plots of the value distribution $p_{\Omega_k}$, by using \eqref{Flowers} with 1000 zeros.
\begin{figure}[!htb]
  \begin{center}
    \includegraphics[width=80mm]{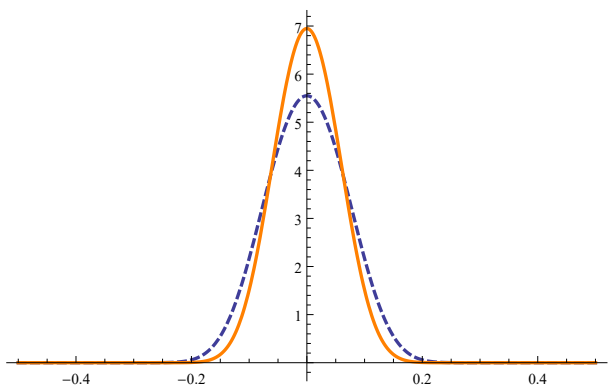}
    \caption{ The value distribution function $p_{\Omega}$ (rescaled) for the superellipse $x^{2k}+y^{2k}\leq 1$, for $k=2$ (dashed) and $k=4$ (solid). }
    \label{fig superellipsedist}
  \end{center}
\end{figure}


\begin{thebibliography}{99}
\bibitem{Bennett}
W.R.  Bennett.
Distribution of the sum of randomly phased components.
\emph{Q. J. Appl. Math.} 5 (1948), 385--393.

\bibitem{BT}
 D.G. Best and T.S.Trudgian.
 Linear relations of zeroes of the zeta-function.
 \emph{Math. Comp.} 84 (2015), no. 294, 2047--2058.

\bibitem{BleherDuke}
P.M. Bleher.
On the distribution of the number of lattice points inside a family of convex ovals.
\emph{Duke Math. J.} 67 (1992), 461--481.

\bibitem{BCDL}
P.M. Bleher,  F.J. Dyson and J.L. Lebowitz.
Non-Gaussian energy level statistics for some integrable systems.
\emph{Phys. Rev. Lett.} 71 (1993), no. 19, 3047--3050.

\bibitem{Btori}
P.M. Bleher.
Distribution of the error term in the Weyl asymptotics for the Laplace operator on a two-dimensional torus and related lattice problems.
\emph{Duke Math. J.} 70 (1993), no. 3, 655--682.


\bibitem{CdV}
Y. Colin de Verdi\`ere.
Nombre de points entiers dans une famille homoth\'etique de domains de $\mathbf R$.
\emph{Ann. Sci. Ecole Norm. Sup.} (4) 10 (1977), no. 4, 559--575.

\bibitem{vdC}
J.G. van der Corput.
Neue zahlentheoretische {A}bsch\"atzungen.
\emph{Math. Ann.} 89 (1923), no. 3-4, 215--254.


 \bibitem{davenport2000multiplicative}
H. Davenport. {\em Multiplicative number theory}. Third edition. Revised and with a preface by Hugh L. Montgomery. Graduate Texts in Mathematics, 74. Springer-Verlag, New York, 2000. xiv+177.

\bibitem{Flower}
J.O. Flower.
Probability distributions of randomly-phased combinations of sine waves of equal amplitude.
\emph{Applied Acoustics} 19 (1986), no. 2, 131--138.

\bibitem{Fujii}
A. Fujii.
An additive problem of prime numbers. II.
\emph{Proc. Japan Acad. Ser. A Math. Sci.} 67 (1991), no. 7, 248--252. An additive problem of prime numbers. III.
\emph{Proc. Japan Acad. Ser. A Math. Sci.} 67 (1991), no. 8, 278--283.

\bibitem{Hlawka}
E. Hlawka.
\"{U}ber die Gleichverteilung gewisser Folgen, welche mit den Nullstellen der Zetafunktionen zusammenh\"{a}ngen.
\emph{Sitzungsber. \"{O}sterr. Akad. Wiss., Math.-Naturw. Kl. Abt. II} 184 (1975), 459--471. 

\bibitem{HB1}
D.R. Heath-Brown.
The distribution and moments of the error term in the Dirichlet divisor problem.
\emph{Acta Arith.} 60 (1992), no. 4, 389--415.


\bibitem{Huxley world record}
M.N. Huxley.
Exponential sums and lattice points. III.
\emph{Proc. London Math. Soc.} (3) 87 (2003), no. 3, 591--609.

\bibitem{Littlewood 1914}
J.E. Littlewood.
Sur la distribution des nombres premiers.
\emph{C. R. Acad. Sci. Paris}  158 (1914), 1869--1872.

\bibitem{Nowak}
 W.G. Nowak.
An $\Omega$-estimate for the lattice rest of a convex planar domain.
\emph{Proc. Roy. Soc. Edinburgh Sect. A} 100 (1985), no. 3-4, 295--299.

\bibitem{Rademacher}
H.A. Rademacher. Fourier Analysis in Number Theory, Symposium on Harmonic Analysis
and Related Integral Transforms (Cornell Univ., Ithaca, N.Y., 1956) in: Collected Papers of Hans Rademacher, Vol. II, pp. 434--458, Massachusetts Inst. Tech., Cambridge, Mass., 1974.

\bibitem{Randol1}
B. Randol.
A lattice-point problem.
\emph{Trans. Amer. Math. Soc.} 121 (1966), 257--268.

\bibitem{Rayleigh}
 L. Rayleigh.
 On the problem of random vibrations, and of random flights in one, two or three dimensions.
\emph{Phil. Mag.} (6) 37 (1919), 321--347.

\bibitem{RubSar}
M.O. Rubinstein and P. Sarnak. Chebyshev's bias.
\emph{Experiment. Math.} 3 (1994), no. 3, 173--197.


\bibitem{Slack}
M. Slack. The probability distributions of sinusoidal oscillations combined in random phase.
\emph{J. Inst. Elect. Eng.}, pt. 3, vol. 93 (1946), pp. 76--86.

\bibitem{Soundararajan}
K. Soundararajan.
Omega results for the divisor and circle problems.
\emph{Int. Math. Res. Not.} (2003), no. 36, 1987--1998.

\bibitem{Tsang}
K.M. Tsang.
Higher-power moments of $\Delta(x)$, $E(t)$ and $P(x)$.
\emph{Proc. London Math. Soc.} (3) 65 (1992), no. 1, 65--84.



\bibitem{WintnerDioph}
A. Wintner.
Upon a Statistical Method in the Theory of Diophantine Approximations.
\emph{Amer. J. Math.}
Vol. 55, No. 1 (1933), pp. 309--331.

\bibitem{Wintner}
A. Wintner.
On the Lattice Problem of Gauss,
\emph{Amer. J. Math.} 63, (1941), 619--627.

\bibitem{Wintner 1941}
A. Wintner. On the distribution function of the remainder term of the prime number theorem.
\emph{Amer. J. Math.} 63, (1941), 233--248.


\end{thebibliography}
\end{document}